\documentclass[11pt,twoside,reqno]{amsart}
\usepackage{amsmath, amsfonts, amsthm, amssymb,amscd, graphicx, amscd}
\usepackage{float,epsf}
\usepackage[english]{babel}
\usepackage{enumerate}
\usepackage{tikz}
\usepackage{hyperref}
 \usepackage{float} 
\usepackage{srcltx}
\usepackage{subfigure}
\usepackage{graphics}
\usepackage{epsfig}
\usepackage{geometry}
\usepackage{verbatim}
\usepackage{mathrsfs}
\newtheorem{theorem}{Theorem}[section]
\newtheorem{lemma}[theorem]{Lemma}

\newtheorem{proposition}[theorem]{Proposition}
\newtheorem{corollary}[theorem]{Corollary}
\newtheorem{remark}[theorem]{Remark}

\numberwithin{equation}{section}
\numberwithin{figure}{section}

\makeatletter
\@namedef{subjclassname@2020}{Mathematics Subject Classification 2020}
\makeatother

\renewcommand{\div}{\mathrm{div}} %divergence
\def\intave#1{\int_{#1}\hbox{\llap{$\raise2.3pt\hbox{\vrule
height.9pt width7pt}\phantom{\scriptstyle{#1}}\mkern-2mu$}}}

\everymath{\displaystyle}

\title{Stability analysis on the thermal insulation problems}
\usepackage[misc]{ifsym}

\author[Y. Huang]{Yong Huang}
\address{School of Mathematics, Hunan University, Changsha, Hunan, China.}
\email{huangyong@hnu.edu.cn}

\author[Q. Li]{Qinfeng Li%\uppertext{*}
}
\email{liqinfeng1989@gmail.com}

\author[Q. Li]{Qiuqi Li }
\email{qli28@hnu.edu.cn}

\thanks{Research of Yong Huang was supported by the National Science Fund for Distinguished Young Scholars (No. 11625103), Tian Yuan Special Foundation (No. 11926317) and Hunan Science and Technology Planning Project(No. 2019RS3016). Both researches of Qinfeng Li and Qiuqi Li were supported by the Fundamental Research Funds for the Central Universities, Hunan Provincial Key Laboratory of intelligent information processing and Applied Mathematics.}

\begin{document}

\begin{abstract}  
 Based on the domain variational point of view, we carry on stability analysis
on two shape optimization problems from thermal insulation background. The novelty is
that, we do not require that the second variation is normal to the boundary. For example,
translation variation is not normal, but as one can see in our work, it not only plays a
role in obtaining the necessary and sufficient condition for stability of ball shape in the
first problem when heat source is radial, but also is essential in deriving the precise value
of symmetry breaking threshold of insulation material in the second problem, which turns
out to be related to isoperimetric constant and in turn implies that ball shapes are stable
in two dimensions.

\end{abstract}

\maketitle

\vskip0.2in

\section{Introduction}
In this paper, we consider two thermal insulation problems, both with the goal of finding the optimal distribution of insulation materials around the boundary of thermal bodies, and then designing optimal shapes, for the purpose of maximizing averaged temperature, or minimizing temperature decay. Let $\Omega \subset \mathbb{R}^n$ be a thermally conducting body, and 
\[
\Sigma_\epsilon:=\{\sigma+t\nu(\sigma)|\sigma\in\partial\Omega\,, 0\le t \le \epsilon h(\sigma)\}
\]
be the layer with thickness of order $\epsilon$, where $h$ is the distribution, after $\epsilon$-normalization, of the insulating material around $\partial \Omega$ with total amount $m>0$, that is,  $$h \in \mathcal{H}_m(\Omega):=\Big\{h \ge 0: \int_{\partial \Omega} h(\sigma) d\sigma=m\Big\}.$$
%We also let $u \in H^1_0(\Omega \cup \Sigma_{\epsilon})$ be the temperature function. 

In the first problem of maximizing averaged temperature as shown in Bucur-Buttazzo-Nitsch\cite{BBN1}, there is a heat source $f \in L^2(\Omega)$, and we also assume that outside $\Omega \cup \Sigma_{\epsilon}$, the temperature is $0$. Hence the long-time temperature $u_h$ inside $\Omega\cup \Sigma_{\epsilon}$, depending on $h$, satisfies
\begin{align*}
    \begin{cases}
    -\Delta u=f\quad &\mbox{in $\Omega$}\\
    -\Delta u=0 \quad &\mbox{in $\Sigma_{\epsilon}$}\\
    u=0\quad &\mbox{on $\partial (\Omega \cup \Sigma_{\epsilon})$}\\
    \frac{\partial u^-}{\partial \nu}=\delta \frac{\partial u^+}{\partial \nu}\quad &\mbox{on $\partial \Omega$, $\delta$ is the conductivity coefficient.}
    \end{cases}
\end{align*}
Equivalently, this solution $u_h$ can be also seen as the solution to the minimization problem among $H^1_0(\Omega \cup \Sigma_{\epsilon})$ class for the energy functional  
\begin{align}
\label{rawenergy}
    E_{\epsilon,\delta,h,\Omega}(u)=\frac{1}{2}\int_{\Omega}|\nabla u|^2dx+\frac{\delta}{2}\int_{\Sigma_\epsilon} |\nabla u|^2dx-\int_{\Omega} fudx.
\end{align}
After integration by parts, the infimum above is given by $-\frac{1}{2}\int_{\Omega}fu_h\, dx$. When $\Omega$ and $f$ are fixed, we would like to find $h \in \mathcal{H}_m(\Omega)$ such that
$
\int_{\Omega}fu_h\, dx
$
is maximized. Hence this is equivalent to minimizing the above energy \eqref{rawenergy}, and thus the first problem is also called the \textit{energy problem}. When $\epsilon<<\delta \rightarrow 0$, we may neglect the insulator, and when $\delta<<\epsilon\rightarrow 0$, there is no heat transmission through $\partial \Omega$. Hence we do not consider these two special cases. For the case $\delta$ is comparable to $\epsilon$, mathematically we may assume that $\delta=\epsilon$, and then due to the work of Acerbi-Buttazzo\cite{AG}, when $\delta=\epsilon\rightarrow 0$, $E_{\epsilon,\delta, m, \Omega}$ $\Gamma$-converges to $$J_m(u,h,\Omega):=\frac{1}{2}\int_{\Omega}|\nabla u|^2dx+\frac{1}{2}\int_{\partial \Omega} \frac{u^2}{h} d\sigma-\int_{\Omega} fudx.$$ 
Then, by H\"older's inequality, see Buttazzo\cite{Buttazzo}, or \cite{BBN} for details, the optimal distribution $h_m$ is given by 
\begin{align}
\label{material}
    h_m(\sigma)=\frac{m|u_{\Omega}(\sigma)|}{\int_{\partial \Omega}|u_{\Omega}|d\sigma},
\end{align}where $u_{\Omega}$ is a minimizer to
\begin{align}
\label{energy1}
J_m(u,\Omega):=\frac{1}{2}\int_{\Omega} |\nabla u|^2dx+\frac{1}{2m}\left(\int_{\partial \Omega} |u| d\mathcal{H}^{n-1}\right)^2-\int_{\Omega} fu dx.
\end{align}
Such $u_{\Omega}$ is unique if $\Omega$ is connected, as shown in \cite{BBN}. 
%We also refer to Brezis-Caffarelli-Friedman\cite{BCF}, Caffarelli-Kriventsov\cite{CK}, Denzler\cite{Den}, Della Pietra-Nitsch-Scala-Trombetti\cite{PNST}, Della Pietra-Nitsch-Trombetti\cite{PNT}, etc, for related analysis.

It is then natural to allow $\Omega$ to vary, and we consider following shape functional in finding minimizers to $E_m(\cdot)$ under volume constraints, for the purpose of designing optimal shapes for maximizing averaged temperature.
\begin{align}
    \label{energyfunctional}
E_m(\Omega):=\inf\Big\{\frac{1}{2}\int_{\Omega}|\nabla u|^2 dx+\frac{1}{2m}\left(\int_{\partial \Omega}|u|d\sigma\right)^2-\int_{\Omega}fudx: u \in H^1(\Omega)\Big\}.
\end{align} 
When $f \equiv 1$, according to Della Pietra-Nitsch-Scala-Trombetti\cite{PNST}, ball is the unique minimizer for any $m>0$. However, for general radial heat source $f$, there is an interesting isoperimetric problem: among all bodies with prescribing volume, is ball also a minimizer to $E_m(\cdot)$ in \eqref{energyfunctional}?

Motivated from the problem, we derive the following stability, instability and stability breaking results under various conditions on $f$.
\begin{theorem}
\label{stabilityintro}
Let $f\ge 0$ be a smooth radial function, and $B_R \subset \mathbb{R}^n$ be the Euclidean ball of radius $R$ centered at origin, then we consider the following three cases.
\begin{itemize}
\item[(1)] If $f$ is nondecreasing along the radius and $f$ is not a constant function, then  $B_R$ is not a stable solution to $E_m(\cdot)$ for any $m>0$. 
\item[(2)] If $f$ is nonincreasing along the radius and further satisfies \begin{align}
\label{conditionforf}
    f(R) \ge \frac{n-1}{n}\frac{\int_{B_R}f(x)dx}{|B_R|},
\end{align}where $f(R)$ is the value of $f$ on $\partial B_R$, then $B_R$ is a stable solution to $E_m(\cdot)$ for any $m>0$.
\item[(3)] If $f'(R)<0$ and $f$ does not satisfy \eqref{conditionforf}, then there exists $m_1>0$ such that when $m<m_1$, $B_R$ is not a stable solution to $E_m(\cdot)$, and when $m>m_1$ $B_R$ is a stable solution to $E_m(\cdot)$.
\end{itemize}
\end{theorem}
The existence of minimizer to \eqref{energyfunctional} is not known, due to lack of uniform perimeter estimates for minimizing sequences. A partial result is given by Du-Li-Wang\cite{DLW}, where they study the existence in the framework of uniform domains, which are roughly extension domains with some uniform parameters. We refer the reader to the two significant papers Gehring-Osgood\cite{GO} and Jones\cite{Jones} for properties of extension domains and uniform domains.

It is easy to see that for any $m>0$ and any such $f$, $B_R$ is always a stationary solution to $E_m(\cdot)$, i.e.
\begin{align}
    \label{stationarycondition}
\frac{d}{dt}\Big|_{t=0}E_m(F_t(\Omega))=0,  
\end{align}for any volume-preserving flow map $F_t$. see Proposition \ref{firstderivative'} below. To clarify the stability analysis, we also derive the second variation formula for $E_m(\cdot)$ at ball shape, and by stability of a domain $\Omega$ we mean that
\begin{align}
    \label{stablecondition}
\frac{d^2}{dt^2}\Big|_{t=0}E_m(F_t(\Omega))\ge 0.
\end{align}
The computation of the second variation is much more complicated since we do not require normal variation. Then by relating to Stekl\"off eigenvalue problem, sphere harmonics and construction of translation variation, we obtain necessary and sufficient conditions on $f$ such that ball shapes are stable to the energy functional, see Theorem~\ref{classification}, and then as a consequence of which we derived Theorem \ref{stabilityintro}. The method in computing second variation of $E_m(\cdot)$ at ball shape is based on some derivative formulas in flow language derived in section 2, which are also used to study the second problem. 

%The condition \eqref{conditionforf} can be viewed as a perturbation from below of $f$ being a positive constant. Thus Theorem \ref{stabilityintro} generalizes the domain variational approach in \cite{DLW}, where the computations were carried out only for normal variations and $f \equiv 1$.

In the second problem of minimizing temperature decay, the domain $\Omega$ is as above, while with a fixed initial temperature $u_0>0$ and there is no heat source. In this case the temperature decays to zero and our goal is to put the insulating material around $\Omega$ in order for the temperature decay rate to be as low as possible. By the Fourier analysis of the following corresponding heat diffusion equations
\begin{align*}
    \begin{cases}
    u_t-\Delta u=0\quad &\mbox{in $\Omega$}\\
    u_t-\Delta u=0 \quad &\mbox{in $\Sigma_{\epsilon}$}\\
    u=0\quad &\mbox{on $\partial (\Omega \cup \Sigma_{\epsilon})$}\\
    \frac{\partial u^-}{\partial \nu}=\delta \frac{\partial u^+}{\partial \nu}\quad &\mbox{on $\partial \Omega$, $\delta$ is the conductivity coefficient}
    \end{cases}
\end{align*}
the decay of the temperature goes as $e^{-t\lambda}$, where $\lambda$ is the first eigenvalue of the following elliptic operator written as
    $$<\mathcal{A}u,\phi>:=\int_{\Omega}\nabla u \nabla \phi\, dx+\delta \int_{\Sigma_{\epsilon}}\nabla u\nabla \phi\, dx.$$
Hence the second thermal insulation problem is also called the \textit{eigenvalue problem}. Similar to the case of the first problem, eventually we need to consider the following minimization problem
\begin{align}
    \label{lambdam'}
\lambda_m(\Omega):=\inf\Big\{\frac{\int_{\Omega}|\nabla u|^2dx+\frac{1}{m}\left(\int_{\partial \Omega}|u|d\sigma\right)^2}{\int_{\Omega}u^2dx}:u\in H^1(\Omega)\Big\}.
\end{align}
Again as before, if $u_{\Omega}$ is the function where the infimum in \eqref{lambdam'} is attained, then the distribution function of insulation material is still given by \eqref{material}. Different from the first problem, the following theorem due to Bucur-Buttazzo-Nitsch shows that that symmetry breaking occurs in the second problem, when the domain is a ball.  
\begin{theorem}\emph{(\cite[Theorem 3.1]{BBN})}
\label{gongming}
Let $\Omega$ be a ball, then there exists a unique number $m_0>0$ such that $\lambda_{m_0}(\Omega)=\mu_2(\Omega)$, the first nonzero eigenvalue of Neumann Laplacian on the ball. Also, when $m>m_0$, the optimal distribution is uniform along the boundary, while when $m<m_0$, the optimal distribution of material is not uniform.
\end{theorem}
However, this theorem does not indicate the precise value of $m_0$. By deriving the second shape derivative of $\lambda_m(\cdot)$, and using translation invariant property, it turns out that the precise value of $m_0$ satisfies the following formula involving isoperimetric constant.
\begin{theorem}
\label{asympintro}
Let  $P(B_R)$ and  $|B_R|$ denote the perimeter and the volume of $B_R$ respectively, then we have
 \begin{align}
\label{buxiedai}
    m_0\mu_2(B_R)=\frac{n-1}{n}\frac{P^2(B_R)}{|B_R|}\, (=2\pi,\,\mbox{when $n=2$.})
\end{align}
\end{theorem}
This formula is rather interesting to us, not only because the number $2\pi$ occurs when $n=2$, but also because it gives a new way of seeing where the symmetry breaking of insulating material occurs: when $n=2$, it occurs when $m\lambda_m(B_R)$ exactly reaches half of its full range for $m \in (0,\infty)$, according to the fact that $m\lambda_m$ is an increasing function, and satisfies
\begin{align}
    \label{range}
\lim_{m\rightarrow \infty}m\lambda_m(B_R)=\frac{P^2(B_R)}{|B_R|}\quad \mbox{and}\quad  \lim_{m\rightarrow 0}m\lambda_m(B_R)=0.    
\end{align} 
Another interesting consequence of \eqref{buxiedai} is that, the value of $m$ at which the symmetry of insulating material around $\partial B_R$ is breaking, is in fact proportional to the volume of $B_R$, instead of the perimeter of $B_R$. It would be very interesting to find a more intuitive perspective to understand \eqref{buxiedai}.

\medskip

By  \eqref{buxiedai} of Theorem~\ref{gongming}, Fourier series and properties of Bessel functions, we prove the following stability result.
\begin{theorem}
When $n=2$, $B_R$ is stable to $\lambda_m(\cdot)$ for any $m>m_0$.
\end{theorem}
We point out that \eqref{buxiedai} is crucial in the proof of the theorem since $m\lambda_m(B_R)-2\pi$ is a factor in the second shape derivative formula at $B_R$. The threshold $m_0$ turns out to be the the stationarity breaking threshold of ball shape in dimension $2$. In fact, as shown in \cite[Theorem 4.1]{BBN}, when $n=2$ and $m<m_0$, any disk cannot be a stationary shape to \eqref{lambdam'} among regions with fixed area. Bucur-Buttazzo-Nitsch conjecture that the disk should be a minimzer when $m>m_0$, and the conjecture remains open so far. Our theorem above says that $B_R$ is a local minimizer to $\lambda_m(\cdot)$, which supports Buttazzo-Bucur-Nitsch's conjecture.  It is noted that when $n \ge 3$, ball is not a minimizer to $\lambda_m(\cdot)$ for any $m>0$.

Last, we remark that when $f \equiv 1$, the first problem is equivalent to solving
\begin{align}
    \label{q1'}
\inf\Big\{\frac{\int_{\Omega}|\nabla u|^2dx+\frac{1}{m}\left(\int_{\partial \Omega}|u|d\sigma\right)^2}{\left(\int_{\Omega}udx\right)^2}:u\in H^1(\Omega)\Big\}.
\end{align}Hence both of the two thermal insulation problems have some similarities with the problem of minimizing eigenvalues of Robin Laplacian. In this direction, we refer to Bucur-Daners\cite{BD10}, Bucur-Giacomini\cite{BG10}-\cite{BG15}, Daners\cite{Daners00}-\cite{Daners06}, etc, for existence and uniqueness results for minimizers among sets with fixed volume. Readers can also see Brasco-De Philippis-Ruffini\cite{BPR12}, Brasco-De Philippis-Velichkov\cite{BPV15},  Fusco-Zhang\cite{FZ17}, Bucur-Giacomini-Nahon\cite{BGN20}, etc, for quantitative estimates in similar problems.

The paper is organized as follows. In section 2, we derive the formulas of some geometric evolution equations on hypersurfaces along smooth flows which may not be perpendicular to the hypersurfaces. Based on the formulas, in section 3 we derive the first domain variation formula for $E_m(\cdot)$. We also derive the second domain variation formula for $E_m(\cdot)$ at ball shape. Then in section 4, based on the second variation formula, we give a necessary and sufficient condition on $f$ such that ball configurations are stable in the first problem, and then we obtain Theorem \ref{stabilityintro}. In section 5, we study the second problem by deriving the first and second derivatives of $\lambda_m(\cdot)$, and then we prove \eqref{buxiedai} and its consequences. In section 6, we prove the stability of ball to $\lambda_m(\cdot)$ for any $m>m_0$ in two dimensions.

\section{Geometric evolution equations on arbitrary flows}

  Geometric evolution equations along normal flows can be found in many references, for example  Huisken-Polden\cite{HP}. When we do first and second variations of integrals over varying domains along normal flows, those formulas are very helpful, as seen from the computations in \cite{DLW}. However, sometimes we need to consider variations along other directions, and then we need to develop general formulas. The computations in this paper partially borrow from the ideas presented in the very inspiring monograph Henrot-Pierre\cite{Henrot}, where the shape derivatives of some classical functionals are computed for the deformation map $F_t(x)=x+t\eta(x)$, where $\eta$ is any smooth vector field. We have also learned from Bucur-Giacomini\cite{BG} the very well written paper by Bandle-Wagner\cite{BW}, where second domain variations for problems with Robin boundary conditions are computed. The volume-preserving map chosen in \cite{BW} is $F_t(x)=x+tV(x)+\frac{t^2}{2}W(x)+o(t^2)$, where both $V$ and $W$ are smooth vector fields. The use of flow map in our work does not lose generality, since we do not require the flow to be perpendicular to the boundary of the domain under consideration. Moreover, the advantage is that the evolution equation formulas stated in the flow language look much simpler, and are similar to those often used in the study of geometric evolution equations along normal flows in Huisken-Polden\cite{HP}.

In this section, we will derive the evolution equations of the $(n-1)$-volume, normal speed and mean curvature of a hypersurface in $\mathbb{R}^n$ along arbitrary flows which may not be perpendicular to the hypersurface. These formulas should have been known, but since we haven not found in literature the corresponding formulas with respect to flows along other directions, we will present the formulas with detailed proofs. 

We first stipulate some notations. Given a vector field $\eta \in C^{2}_0(\mathbb{R}^n,\mathbb{R}^n)$ and a hypersurface $M \subset \mathbb{R}^n$, using local coordinates $(x^1,\cdots, x^{n-1})$ of $M$, the tangential gradient of $\eta$ along $M$ is defined as an $n\times n$ matrix
\begin{align}
\label{tangentialgradient}
    \nabla^M \eta =g^{ij}\partial_i \eta \otimes \partial_jF,
\end{align}where $F$ is the position vector for points on $M$, $\partial_iF=\frac{\partial F}{\partial x^i}$, $g_{ij}=<\partial_i F, \partial_jF>$, $g^{ij}$ is the inverse matrix of $g_{ij}$ and $\partial_i \eta=(\nabla \eta)\partial_i F$. Then the tangential divergence of $\eta$ is defined as the trace of $\nabla^M \eta$, that is,
\begin{align}
\label{tangentialdivergence}
    \div_{M}\eta=g^{ij}\partial_i\eta \cdot \partial_jF,
\end{align}where $\cdot$ is the inner product in $\mathbb{R}^n$ and repeated indices means summations. Sometimes we also use $<\cdot,\cdot>$ to denote inner product in $\mathbb{R}^n$. We also adopt the convention that given a function $f$ defined in $\mathbb{R}^n$, $\nabla_i\nabla_jf$ denotes $\nabla^2 f(\partial_i F,\partial_j F)$, where $\nabla^2 f$ is the Hessian of $f$ on $M$, and $\partial_i\partial_j f$ denotes usual derivatives of $f$ first along $\partial_j F$ and then along $\partial_i F$.

We call $F_t(x):=F(t,x)$ is the \textit{flow map} generated by the $C^2$ vector field $\eta$, if 
\begin{align*}
    \frac{\partial F(t,x)}{\partial t}=\eta \circ F(t,x);\quad F(0,x)=x.
\end{align*}
Hence $F_t$ is a $C^1$ family of $C^2$ diffeomorphism, when $|t|$ is small.

%If we run the flow $F_t$ along the hypersurface $M$, and if $\eta \perp \partial M$, then by \cite{Huisken}
%\begin{align}
%\label{evolutionofarea1}
%    \frac{d}{dt}\Big|_{t=0}d\sigma_t=(\eta\cdot \nu) Hd\sigma,
%\end{align}where $\sigma_t$ is the $(n-1)$-volume of $F_t(M)$, $H$ is the mean curvature of $M$, which is given by $-\sum_{i,j=1}^{n-1}g^{ij}<\partial_i\partial_jF,\nu>$. Also, \begin{align}
%    \label{evolutionofH1}
%\frac{d}{dt}\Big|_{t=0}H=-\Delta_M  (\eta\cdot \nu)-|A|^2(\eta \cdot \nu),   
%\end{align}where $\Delta_M$ is the Laplacian on $M$, and $A$ is the second fundamental form.
%\begin{remark}
%\eqref{evolutionofarea1} and \eqref{evolutionofH1} are generally not true for $t\ne 0$, since $\eta\circ F_t$ may not be perpendicular to $\partial F_t(M)$.
%\end{remark}

The following proposition collects some derivative formulas used later.
\begin{proposition}
\label{geometricevolution}
Let $F_t(x):=F(t,x)$ be the flow map generated by a smooth vector field $\eta$, $M_t=F_t(M)$, $\sigma_t$ be the volume element of $M_t$, $\nu(t)$ be the unit normal field along $M_t$ and $h(t)$ be the second fundamental form of $M_t$,  then we have
\begin{align}
    \label{evolutionofarea2}
\frac{d}{dt}d\sigma_t=(\div_{M_t} \eta )d\sigma_t,
\end{align}
\begin{align}
    \label{evolutionofnormalspeed}
\frac{d}{dt}\left(\eta(F_t)\cdot \nu(t)\right)=(\eta(F_t)\cdot\nu(t)) (\div \eta -\div_{M_t}\eta)\circ F_t,
\end{align}and
\begin{align}
    \label{hijevolution}
h_{ij}'(t)=-<\nabla_i\nabla_j\eta,\nu(t)>, 
\end{align}where $h_{ij}(t)$ and $\nabla_i\nabla_j \eta$ are the $i,j$-components of $h(t)$ and the Hessian of $\eta$ on $M_t$, respectively, under local coordinates of $M_t$.

If $M$ is an $(n-1)$-sphere of radius $R$, then we also have
\begin{align}
    \label{evoonsphere}
\frac{d}{dt}\Big|_{t=0}H=-\Delta_M(\eta\cdot \nu)-\frac{n-1}{R^2}\eta\cdot\nu.     
\end{align}
\end{proposition}
\begin{proof}
Let $g=g(t)$ be the metric of $M_t$ induced by $F_t$, and let $F$ be the position vector of $M$. We also let $(x^1, x^2, \cdots, x^{n-1})$ be local coordinates of $M$, and thus they also serve as local coordinates of $M_t$. When taking derivatives of a vector field $\eta$ restricted on $M_t$ with respect to local coordinates, we use the abbreviation that 
\begin{align}
    \label{convention}
\partial_i \eta=(\nabla \eta \circ F_t) \partial_i F_t,
\end{align}
where $\partial_i F_t=\frac{\partial F_t}{\partial x^i}$, and then we have
\begin{align*}
    g_{ij}'(t)=\frac{d}{dt}<\partial_iF_t,\partial_jF_t>=\partial_i \eta \cdot \partial_j F_t+\partial_iF_t\cdot \partial_j\eta.
\end{align*}
This and \eqref{tangentialdivergence} imply that
\begin{align*}
    \frac{d}{dt}\sqrt{|g(t)|}=\frac{1}{2\sqrt{|g(t)|}}|g(t)|g^{ij}(t) \, 2(\partial_i \eta \cdot \partial_j F_t+\partial_iF_t\cdot \partial_j\eta)=(\div_{M_t}\eta)\sqrt{|g(t)|}. 
\end{align*}Hence \eqref{evolutionofarea2} is proved.

Next, let $\eta^T(F_t)$ be the tangential component of $\eta(F_t)$ along $ M_t$, then we have
\begin{align}
\label{etaT}
    \eta^T(F_t)=g^{ij}(t)<\eta(F_t),\partial_jF_t>\partial_iF_t.
\end{align}Therefore, by \eqref{convention}, \eqref{etaT} and since $\nu'(t)\perp M_t$, we have
\begin{align*}
    \frac{d}{dt}\left(\eta(F_t)\cdot \nu(t)\right)=&\Big((\nabla \eta) \eta\Big)\circ F_t\cdot \nu(t)+<\eta(F_t), g^{ij}(t)\Big( \nu'(t)\cdot \partial_iF_t\Big) \partial_j F_t>\\
    =&\Big((\nabla \eta) \eta\Big)\circ F_t\cdot \nu(t)-g^{ij}(t)<\nu(t),\partial_i \eta><\eta(F_t),\partial_j F_t>\\
    =&\Big((\nabla \eta) \eta^T\Big)\circ F_t\cdot \nu(t)+\Big((\nabla \eta\circ F_t) ((\eta\circ F_t\cdot \nu(t))\nu(t) ) \Big)\cdot \nu(t)-<\nu(t),\Big((\nabla \eta) \eta^T\Big)\circ F_t>\\
    =&\Big((\nabla \eta\circ F_t) ((\eta\circ F_t\cdot \nu(t))\nu(t) ) \Big)\cdot \nu(t)\\
    =&(\eta(F_t)\cdot\nu(t)) (\div \eta -\div_{M_t}\eta)\circ F_t,
\end{align*}where the last equality is obtained by taking the trace of the following identity.
\begin{align*}
    \nabla \eta =\nabla ^{M_t}\eta +\nabla \eta \nu(t)\otimes \nu(t).
\end{align*}
Hence \eqref{evolutionofnormalspeed} is proved.
Next, we compute the derivative of $h(t)$. We have
\begin{align*}
    h_{ij}'(t)=&\partial_i\nu'(t)\cdot \partial_j F_t+\partial_i\nu(t)\cdot \partial_j\eta\\
    =&\partial_i(\nu'(t)\cdot \partial_j F_t)-\nu'(t)\cdot \partial_i\partial_jF_t+\partial_i\nu(t)\cdot \partial_j\eta\\
    =&-\partial_i(\nu(t)\cdot \partial_j\eta)-g^{kl}(t)(\nu'(t)\cdot\partial_kF_t)<\partial_lF_t,\partial_i\partial_jF_t>+\partial_i\nu(t)\cdot \partial_j\eta\\
    =&-<\nu(t),\partial_i\partial_j\eta>-g^{kl}(t)(\nu'(t)\cdot\partial_k F_t) \Gamma_{ij}^m(t)g_{lm}(t), \, \mbox{where $\Gamma_{ij}^m(t)$ are Christoff symbols on $M_t$}\\
    =&-<\nu(t),\partial_i\partial_j\eta>+\Gamma_{ij}^k(t)<\nu(t),\partial_k\eta>\\
    =&-<\nabla_i\nabla_j\eta,\nu(t)>.
\end{align*}Hence \eqref{hijevolution} is proved.

Last, to prove \eqref{evoonsphere}, note first that
\begin{align*}
    H'(t)=&\frac{d}{dt}\left(g^{ij}(t)h_{ij}(t)\right)\\
    =&-2g^{il}(t)g^{mj}(t)<\partial_m\eta,\partial_lF_t>h_{ij}(t)-g^{ij}(t)<\nabla_i\nabla_j\eta,\nu(t)>\\
    =&-2g^{il}(t)g^{mj}(t)<\partial_m\eta,\partial_lF_t>h_{ij}(t)-<\Delta_{M_t}\eta,\nu(t)>.
\end{align*}Since on sphere, $h_{ij}=\frac{g_{ij}}{R}$, using \eqref{tangentialdivergence}, we have
\begin{align}
\label{l0}
    H'(0)=-\frac{2}{R}\div_{M}\eta-<\Delta_M\eta,\nu>.
\end{align}Since $<\Delta_M\eta,\nu>$ does not depend on the choice of coordinates, in the following we choose normal coordinates to do the computation. Let $\eta=\eta^T+\zeta\nu$, then using normal coordinates we have
\begin{align}
\label{l1}
    <\Delta_M\eta^T,\nu>=&\partial_i<\partial_i\eta^T,\nu>-<\partial_i\eta^T,\partial_i\nu>\nonumber\\
    =&-\partial_i<\eta^T,\partial_i\nu>-h_{i}^l<\partial_i\eta^T,\partial_lF>\nonumber\\
    =&-\partial_i\left(h_{i}^l<\eta^T,\partial_lF>\right)-h_{i}^l<\partial_i\eta^T,\partial_lF>\nonumber\\
    =&\frac{1}{R}\left(-\partial_i\left(g_{i}^l<\eta^T,\partial_lF>\right)-g_{i}^l<\partial_i\eta^T,\partial_lF>\right),\quad \mbox{since $\frac{g_{ij}}{R}=h_{ij}$}\nonumber\\
    =&\frac{1}{R}\left(-\partial_i<\eta^T,\partial_iF>-<\partial_i\eta^T,\partial_iF>\right),\quad\mbox{since $g_j^l=\delta_{jl}$ and $\partial_ig_j^{l}=0$}\nonumber\\
    =& \frac{1}{R}\left(-2\div_{M}\eta^T-<\eta^T,\Delta_MF>\right)\nonumber\\
    =&-\frac{2}{R}\div_M\eta^T,\quad \mbox{since $\Delta_MF=-H\nu \perp \eta^T$.}
\end{align}Direct computation also leads to
\begin{align}
\label{l2}
    <\Delta_M(\zeta\nu),\nu>=\Delta_M\zeta-\frac{n-1}{R^2}\zeta.
\end{align}
Since on the sphere of radius $R$, 
\begin{align*}
    \div_{M}\eta=\div_{M}\eta^T+\frac{n-1}{R}\zeta,
\end{align*}by \eqref{l0}-\eqref{l2} we obtain \eqref{evoonsphere}.
\end{proof}

\section{First and Second Variation of $E_m(\cdot)$}
Recall that if $\Omega$ is connected, we use $u_{\Omega}$ to denote the unique minimizer to $J_m(\cdot, \Omega)$, where $J_m(\cdot, \Omega)$ is given by \eqref{energy1}. First, we have
\begin{proposition}
\label{radial}
If $\Omega=B_R$ and $f$ is radial and nonnegative, then $u_{\Omega} \ge 0$ is also radial and satisfies
\begin{align}
\label{ul}
    \begin{cases}
    -\Delta u=f \quad & \mbox{in $\Omega$}\\
    \frac{\partial u}{\partial \nu}=-\frac{1}{m}\int_{\partial \Omega}u d\sigma\quad & \mbox{on $\partial \Omega$}
    \end{cases}
\end{align}
\end{proposition}

\begin{proof}
Note that the solutions to 
\begin{align}
\label{el'}
    \begin{cases}
    -\Delta u=f \quad & \mbox{in $\Omega$}\\
    \frac{\partial u}{\partial \nu}=-\frac{\int_{\Omega}f dx}{P(\Omega)}\quad & \mbox{on $\partial \Omega$}
    \end{cases}
\end{align}are radial, and they are equal up to a constant. Let $u_0$ be a solution to \eqref{el'}, then we can find a constant $c$ such that 
\begin{align*}
    \frac{1}{m}\int_{\partial \Omega} (u_0+c) d\sigma=\frac{\int_{\Omega}f dx}{P(\Omega)}.
\end{align*}Hence $u:=u_0+c$ is a radial solution to \eqref{ul}. Since $f \ge 0$, by \eqref{el'} we know $\frac{\partial u}{\partial \nu} \le 0$ on $\partial \Omega$. Then by $\eqref{ul}_2$ we have $u \ge 0$ on $\partial \Omega$, and $u>0$ if $f \ne0$. Hence $u$ satisfies the Euler Lagrange equation to $J_m(\cdot, \Omega)$, and thus $u$ is a minimizer. Hence $u=u_{\Omega}$ by uniqueness.
\end{proof}
\begin{remark}
If the radial function $f\ge 0$ and $f \ne 0$, then $u_{B_R}>0$ on $\partial B_R$.
\end{remark}
In the following, we write $\Omega_t=F_t(\Omega)$, $u(t)(x):=u_{\Omega_t}(x)$ and $u'(t)(x)=\frac{\partial}{\partial t}u_{\Omega_t}(x)$, where  $F_t(x):=F(t, x)$ is the flow map generated by the velocity field $\eta \in C_0^{\infty}(\mathbb{R}^n,\mathbb{R}^n)$. Note that $u'(t)$ is well defined for $|t|$ small enough, due to similar argument in \cite{Henrot}, \cite{BW} or \cite{DLW}, see also \eqref{buchong} below. Since $F_t$ preserves the volume of $\Omega$, we have
\begin{align}
\label{volumepreserving}
    \int_{\partial \Omega} (\eta \cdot \nu)\, d\sigma=0\,\, \mbox{and}\,\int_{\partial \Omega} (\eta\cdot \nu)\div\eta \, d\sigma=0.
\end{align}
Now we are ready to derive the first variation of energy. 
\begin{proposition}
\label{firstderivative'}
Let $F_t(x):=F(t,x)$ be the flow map generated by $\eta \in C^3_0(\mathbb{R}^n,\mathbb{R}^n)$, and $f\ge 0, f\ne 0$ be a smooth radial function. If $u_{\Omega}>0$ on $\partial \Omega$, then we have
\begin{align}
\label{firstderivative}
    \frac{d}{dt}E_m(\Omega_t)=\int_{\partial \Omega_t}\Big(\frac{1}{2}|\nabla u(t)|^2-\Big|\frac{\partial u(t)}{\partial\nu}\Big|^2+\left(\frac{1}{m}\int_{\partial \Omega_t}u(t) d\sigma \right) u(t)H(t)-fu(t)\Big)\eta \cdot \nu(t) \,d\sigma,
\end{align}where $H(t)$ is the mean curvature of $\partial \Omega_t$ and $\nu(t)$ is the outer unit normal to $\partial \Omega_t$. In particular, if $\Omega$ is a ball, then \eqref{firstderivative} vanishes at $t=0$.
\end{proposition}
\begin{proof}
We let $E_m(\Omega_t)=J_m(u_{\Omega_t},\Omega_t)=I_1(t)+I_2(t)+I_3(t)$ coresponding to the first, second and third term of $J_m$ defined in \eqref{energy1}. We have
\begin{align*}
    I_1'(t)=& \int_{\Omega_t} \nabla u(t) \cdot\nabla u'(t)\,dx+\int_{\partial \Omega_t} \frac{1}{2}|\nabla u(t)|^2 \eta \cdot \nu(t) \,d\sigma_t\\
    =& \int_{\Omega_t} fu'(t)\, dx+\int_{\partial \Omega_t} \frac{\partial u(t)}{\partial \nu }u'(t) d\sigma_t  +\int_{\partial \Omega_t} \frac{1}{2}|\nabla u(t)|^2 \eta \cdot \nu(t) \,d\sigma_t.
\end{align*}
\begin{align*}
    I_2'(t)=&\frac{1}{m}\int_{\partial \Omega_t}u(t) d\sigma_t  \int_{\partial \Omega_t} (u'(t)+\nabla u(t) \cdot \eta  +u(t)\div_{\partial \Omega_t}\eta)d\sigma_t\\
    =&\frac{1}{m}\int_{\partial \Omega_t}u(t) d\sigma_t  \int_{\partial \Omega_t} (u'(t)+\frac{\partial u(t)}{\partial \nu} \eta \cdot \nu(t) +uH(t) \eta \cdot \nu(t))d\sigma_t .
\end{align*}
\begin{align*}
    I_3'(t)=-\int_{\Omega_t} fu'(t)dx-\int_{\partial \Omega_t}fu(t)\eta \cdot \nu(t) d\sigma_t.
\end{align*}
In the above, we have used the formula \eqref{evolutionofarea2},  integration by parts on closed hypersurfaces and the formula
\begin{align}
\label{tibianfen}
    \frac{d}{dt}\int_{F_t(\Omega)}g(t,y)dy=\int_{F_t(\Omega)}g_t(t,y)dy+\int_{\partial F_t(\Omega)}g(t,y)\eta \cdot \nu(t) d\sigma_t.
\end{align}
Since $f\ne 0$, $u$ is strictly positive along $\partial B_R$, and hence if $t$ is small, $u(t)>0$ on $\partial \Omega_t$. Therefore, by applying \eqref{ul} on $(u(t),\Omega_t)$ and adding up $I_1'(t)$, $I_2'(t)$ and $I_3'(t)$, we obtain \eqref{firstderivative}. 

If $\Omega$ is a ball and $f$ is radial, then by Proposition \ref{radial}, $u$ is radial, and hence by  $\eqref{volumepreserving}_1$, \eqref{firstderivative} vanishes when $t=0$.
\end{proof}

Next, we will derive the second variation formula of $E_m(\cdot)$ at $B_R$ under the flow map which is not necessarily normal to $\partial B_R$, based on the formulas proved in section 2.
\begin{lemma}
\label{secondvariation}
Let $\Omega=B_R$, $\Omega_t=F_t(\Omega)$ and $u(t)=u_{F_t(\Omega)}$, where $F_t(x):=F(t,x)$ is the flow map generated by a smooth velocity field $\eta$. Let $f(x)=f(r)$ be a smooth nonnegative radial function and assume $f\ne 0$, then we have 
\begin{align}
\label{secondvariationformula}
   \frac{d^2}{dt^2}\Big|_{t=0}E_m(\Omega_t)   =& \int_{\partial B_R}\left(v\zeta u_{rr}+u_{rr}u_r\zeta^2-f'(r) u\zeta^2\right) d\sigma\nonumber\\
   &+ \left(\frac{1}{m}\int_{\partial B_R} u^2 d\sigma \right)\int_{\partial B_R}\left(|\nabla_{\partial B_R} \zeta |^2-\frac{n-1}{R^2}\zeta^2\right)d\sigma,
\end{align}where $\zeta=\eta \cdot \nu$ on $\partial B_R$ and $v=u'(0)$.
\end{lemma}
\begin{proof}
We let $u=u_{B_R}$ and we also define $$\zeta(t)(x):=\eta(x)\cdot \nu(t),\quad \mbox{on $\partial \Omega_t$},$$ where $\nu(t)$ is the unit outer normal to $\partial \Omega_t$. Hence $\zeta=\zeta(0)$. We first show that
\begin{align}
\label{vequation}
    \begin{cases}
    \Delta v=0 \quad & \mbox{in $B_R$}\\
    \frac{\partial v}{\partial \nu}=-u_{rr}\zeta-\frac{1}{m}\int_{\partial B_R}vd\sigma \quad & \mbox{on $\partial B_R$}
    \end{cases}
\end{align}
Indeed, $\eqref{vequation}_1$ is easy to see. To see $\eqref{vequation}_2$, first note that by Proposition \ref{radial}, $u$ is radial, and hence the Euclidean hessian of $u$ is given by
\begin{align}
    \label{hessianu}
 \nabla^2 u(x)=u_{rr}\frac{x}{|x|}\otimes\frac{x}{|x|}+\frac{u_r}{r}(I-\frac{x}{|x|}\otimes \frac{x}{|x|}),
\end{align}where $I$ is the identity matrix. Hence given a test function $\phi$, by $\eqref{volumepreserving}_1$, \eqref{evolutionofarea2} and \eqref{hessianu} we have
\begin{align*}
   0=&\frac{d}{dt}\Big|_{t=0} \int_{\partial \Omega_t}\left(\frac{1}{m}\int_{\partial \Omega_t}u(t) d\sigma_t +\frac{\partial u(t)}{\partial \nu }\right)\phi d\sigma_t\\
   =&\int_{\partial B_R} \phi\Big(\frac{1}{m}\big(\int_{\partial B_R}(v+u_r\zeta +uH\zeta ) d\sigma \big)+\frac{\partial v}{\partial \nu}+u_{rr}\zeta\Big)d\sigma\nonumber\\
   & +\int_{\partial B_R}\left(\frac{1}{m}\int_{\partial B_R}u d\sigma +\frac{\partial u}{\partial \nu }\right)(\nabla \phi\cdot \eta+\phi\div_{\partial B_R}\eta) d\sigma\\
   =& \int_{\partial B_R}\phi\left(\frac{1}{m}\int_{\partial B_R}vd\sigma+\frac{\partial v}{\partial \nu}+u_{rr}\zeta\right)d\sigma,
\end{align*}where we have used the tangential divergence theorem on hypersurface, that $u(t)$ is a solution to \eqref{ul}, that $u$ is radial in $B_R$ and that on $\partial B_R$ we have 
\begin{align}
\label{xuyubuli}
    \nabla u \cdot \eta=u_r\zeta
\end{align}and
\begin{align*}
    <(\nabla^2u)\eta, \nu>=\nabla^2u :\eta \otimes \nu=u_{rr}\zeta.
\end{align*} 
Therefore,  $\eqref{vequation}_2$ holds since $\phi$ is arbitrary.

By $\eqref{volumepreserving}_1$  and the compatibility condition for the existence of solutions to \eqref{vequation}, we have 
\begin{align}
\label{stekloffcondition}
    \frac{1}{m}\int_{\partial B_R}v d\sigma =0.
\end{align}Hence by \eqref{vequation}, $v$ satisfies
\begin{align}
    \label{vequation'}
    \begin{cases}
    \Delta v=0 \quad & \mbox{in $B_R$}\\
    \frac{\partial v}{\partial \nu}=-u_{rr}\zeta\quad & \mbox{on $\partial B_R$.}
    \end{cases}
\end{align}
Also, from $\eqref{volumepreserving}_1$ and \eqref{stekloffcondition}, we have
\begin{align}
    \label{derriave}
\frac{d}{dt}\Big |_{t=0}\left(\frac{1}{m}\int_{\partial \Omega_t} u(t) d\sigma_t \right)=0.
\end{align}
Now we do the second variation based on \eqref{firstderivative}. Taking one more derivative, by $\eqref{volumepreserving}_2$, \eqref{evolutionofarea2}, \eqref{evolutionofnormalspeed} and $\eqref{vequation'}_2$ and since $u$ is radial, we have
\begin{align}
\label{1}
    &\frac{d}{dt}\Big|_{t=0}\int_{\partial \Omega_t}\frac{1}{2}|\nabla u(t)|^2 \zeta(t) d\sigma_t\nonumber\\
    =&\int_{\partial B_R}\left( u_rv_r\zeta+u_ru_{rr}\zeta^2+\frac{1}{2}|\nabla u|^2\zeta( \div \eta-\div_{\partial B_R}\eta)+\frac{1}{2}|\nabla u|^2 \zeta\div_{\partial B_R}\eta\right)d\sigma \nonumber\\
    =&\int_{\partial B_R}\left(u_r\zeta (v_r+u_{rr}\zeta)+\frac{1}{2}|\nabla u|^2\zeta \div \eta\right)d\sigma=0,
\end{align}where we have also used that 
\begin{align*}
\nabla u\cdot \nabla v=\partial_{\nu} v \nabla u \cdot \nu=u_rv_r,
\end{align*}
\begin{align*}
    \nabla u  \cdot (\nabla^2 u\,  \eta)=u_r(\nabla^2u:\eta\otimes \nu)=u_ru_{rr}\zeta,
\end{align*}and the formula \eqref{evolutionofnormalspeed}.
Similarly as above, we have 
\begin{align}
\label{2}
    \frac{d}{dt}\Big|_{t=0}\int_{\partial \Omega_t}|\nabla_{\nu} u(t)|^2 \zeta(t) d\sigma_t=0.
\end{align}
\eqref{2} can also be computed by replacing $\partial_{\nu} u(t)$ with $\frac{1}{m}\int_{\partial \Omega_t}u d\sigma_t$, then $\eqref{volumepreserving}_2$ \eqref{evolutionofarea2}, \eqref{evolutionofnormalspeed} and \eqref{derriave} yield
\begin{align*}
    \frac{d}{dt}\Big|_{t=0}\int_{\partial \Omega_t}|\nabla_{\nu} u(t)|^2 \zeta(t) d\sigma_t=\int_{\partial B_R}u_r^2\left( \zeta (\div \eta-\div_{\partial B_R}\eta)+\zeta \div_{\partial B_R}\eta)\right)d\sigma=0.
\end{align*}

Next, by $\eqref{volumepreserving}_2$, \eqref{evolutionofarea2}, \eqref{evolutionofnormalspeed}, \eqref{evoonsphere}, \eqref{xuyubuli}, \eqref{derriave}, and since $u$ is radial, we have
\begin{align}
    \label{3}
&\frac{d}{dt}\Big|_{t=0}\int_{\partial \Omega_t}  \left(\frac{1}{m}\int_{\partial \Omega_t} u(t) d\sigma_t \right)u(t)H\zeta(t) d\sigma_t\nonumber\\
=& \left(\frac{1}{m}\int_{\partial B_R} u d\sigma \right)\int_{\partial B_R}\Big(vH\zeta+u_r\zeta H\zeta+u(-\Delta_{\partial \Omega_t} \zeta-\frac{n-1}{R^2}\zeta)\zeta \nonumber\\
&\quad \quad \quad \quad \quad \quad \quad \quad \quad +uH\zeta( \div \eta -\div_{\partial B_R}\eta)+uH\zeta \div_{\partial B_R}\eta\Big)d\sigma\nonumber\\
=& \left(\frac{1}{m}\int_{\partial B_R} u d\sigma \right)\int_{\partial B_R} (vH\zeta+u_rH\zeta^2)d\sigma\nonumber\\
&+ \left(\frac{1}{m}\int_{\partial B_R} u^2 d\sigma \right)\int_{\partial B_R}\left(|\nabla_{\partial B_R} \zeta |^2-\frac{n-1}{R^2}\zeta^2\right)d\sigma.
\end{align}
Similarly as above, we can compute
\begin{align}
    \label{4}
&\frac{d}{dt}\Big|_{t=0}\int_{\partial \Omega_t} fu(t)\zeta(t) d\sigma\nonumber \\
=&\int_{\partial B_R} \left((\nabla f\cdot \eta) u\zeta+fv\zeta+f\nabla u\cdot \eta\zeta +fu\zeta (\div \eta-\div_{\partial B_R}\eta)+fu\zeta \div_{\partial B_R}\eta\right)d\sigma\nonumber\\
=&\int_{\partial B_R}\left( (f_ru+fu_r)\zeta^2+fv\zeta\right)d\sigma, \quad \mbox{since $\eqref{volumepreserving}_2$.}
\end{align}
Combining \eqref{1}-\eqref{4}, and since $\frac{1}{m}\int_{\partial B_R}u d\sigma=-u_r$, $H=\frac{n-1}{R}$ on $\partial B_R$ and $u_{rr}+\frac{n-1}{R}u_r+f=0$, we have
\begin{align}
\label{guodu2}
    \frac{d^2}{dt^2}\Big|_{t=0}E_m(\Omega_t) =& \int_{\partial B_R} u_{rr}u_r\zeta^2 d\sigma +\int_{\partial B_R}v\zeta u_{rr}d\sigma-\int_{\partial B_R}f_ru\zeta^2d\sigma \nonumber\\
    &+ \left(\frac{1}{m}\int_{\partial B_R} u^2 d\sigma \right)\int_{\partial B_R}\left(|\nabla_{\partial B_R} \zeta |^2-\frac{n-1}{R^2}\zeta^2\right)d\sigma.
\end{align}Hence we have proved \eqref{secondvariationformula}.
\end{proof}

\section{Necessary and Sufficient conditions for stability of ball configurations}
In this section, we still assume that $f \ge 0$ is a radial function and we denote $f(x)$ by $f(r)$, where $r=|x|$. Let $\bar{f}_R=\frac{\int_{B_R} f(x)dx}{|B_R|}$, then we have
\begin{theorem}
\label{classification}
If $f \ge 0$ is radial and satisfies
\begin{align}
    \label{sufnecconditionforf}
(f-\frac{n-1}{n}\bar{f}_R)(f-\bar{f}_R)+f'(R)\bar{f}_R\frac{m}{n^2\omega_nR^{n-1}}\le0 \quad \mbox{on $\partial B_R$,}
\end{align}then $B_R$ is stable along any volume preserving flows. The converse is also true.
\end{theorem}
\begin{proof}
Let $\eta \in C_0^3(\mathbb{R}^n,\mathbb{R}^n)$ be the velocity field of the volume preserving flow starting from $\Omega$.
Since the first eigenvalue of Laplacian on the unit sphere is $(n-1)$, it follows from \eqref{secondvariationformula} that 
\begin{align}
\label{guodu}
    \frac{d^2}{dt^2}\Big|_{t=0}E_m(F_t(B_R))\ge \int_{\partial B_R} u_{rr}u_r\zeta^2 d\sigma +\int_{\partial B_R}v\zeta u_{rr}d\sigma-\int_{\partial B_R}f'(r)u\zeta^2d\sigma,
\end{align}where $u=u_{B_R}$ and $v$ is a solution to \eqref{vequation'}. Since \eqref{stekloffcondition} and that the second Stekloff eigenvalue on $B_R$ is $\frac{1}{R}$, we have that
\begin{align}
    \label{useful}
\int_{\partial B_R}v^2 d\sigma \le R\int_{B_R}|\nabla v|^2\,dx.   
\end{align}Hence
\begin{align*}
   \left( \int_{\partial B_R} v\zeta d\sigma\right)^2 \le& \int_{\partial B_R}v^2 d\sigma \int_{\partial B_R}\zeta ^2 d\sigma \\
   \le& R\int_{B_R}|\nabla v|^2dx \int_{\partial B_R}\zeta ^2 d\sigma \\
   =& R\int_{\partial B_R} vv_rd\sigma \int_{\partial B_R}\zeta^2 d\sigma =-R\int_{\partial B_R}u_{rr}v\zeta d\sigma\int_{\partial B_R}\zeta^2.
\end{align*}
Since \begin{align*}
    -\int_{\partial B_R}u_{rr}v\zeta d\sigma=\int_{\partial B_R}v_rv d\sigma=\int_{B_R}|\nabla v|^2 dx \ge 0,
\end{align*}from the above we have that
\begin{align*}
    -\int_{\partial B_R}u_{rr}v\zeta d\sigma \le R\int_{\partial B_R}u_{rr}^2 \zeta^2 d\sigma.
\end{align*}Hence \eqref{guodu} yields
\begin{align}
\label{keyestimate1}
     \frac{d^2}{dt^2}\Big|_{t=0}E_m(F_t(B_R)) \ge& \int_{\partial B_R}(u_{rr}u_r\zeta^2-Ru_{rr}^2\zeta^2)d\sigma -\int_{\partial B_R}f_ru\zeta^2\nonumber\\
     =& R\int_{\partial B_R}\zeta^2 u_{rr}(f+\frac{n}{R}u_r)d\sigma-\int_{\partial B_R}f_r u\zeta^2 d\sigma.
\end{align} 
Since on $\partial B_R$, we have
\begin{align*}
    u_r=-\frac{\int_{B_R}f dx}{P(B_R)}=-\frac{R}{n}\bar{f}_R,
\end{align*}
\begin{align*}
    u=-\frac{mu_r}{P(B_R)}=\frac{m\int_{B_R}f dx}{P^2(B_R)}=\frac{m\bar{f}_R}{n^2\omega_nR^{n-2}},
\end{align*}and
\begin{align*}
    u_{rr}=-f-\frac{n-1}{R}u_r=-f+\frac{n-1}{n}\bar{f}_R,
\end{align*}
from \eqref{keyestimate1} we have that 
\begin{align}
\label{xiajie}
    \frac{d^2}{dt^2}\Big|_{t=0}E_m(F_t(B_R)) \ge -R\int_{\partial B_R}\left((f-\frac{n-1}{n}\bar{f}_R)(f-\bar{f}_R)+f_r\bar{f}_R\frac{m}{n^2\omega_nR^{n-1}}\right)\zeta^2d\sigma.
\end{align}
Hence if $f$ satisfies \eqref{sufnecconditionforf}, then $ \frac{d^2}{dt^2}\Big|_{t=0}J(u(t),\Omega_t) \ge 0$ and thus $(u_{B_R},B_R)$ is stable.

Conversely, if $(u_{B_R},B_R)$ is stable along any volume preserving flows, then in particular this is true for translations with constant speed. That is, 
\begin{align*}
    \frac{d^2}{dt^2}\Big|_{t=0}E_m(F_t(B_R)) \ge 0,
\end{align*}for $F_t(B_R)=\{x+t\eta:x\in B_R\}$, where $\eta$ is a constant vector field. Let $\eta=(c_1,\cdots, c_n)^T$, thus \begin{align}
    \label{shougen0}
    \zeta=\frac{1}{R}\sum_{i=1}^nc_ix_i.
\end{align}
Hence we can find general solutions to \eqref{vequation'}, which is
\begin{align}
\label{shougen1}
    v(x)=-u_{rr}(R)\sum_{i=1}^nc_ix_i+C.
\end{align}
For the choice of $F_t$, since $\zeta$ is now the first eigenfunction of Laplacian on $\partial B_R$, applying \eqref{secondvariationformula} we have that
\begin{align}
\label{wohewodejiaxiang}
    \frac{d^2}{dt^2}\Big|_{t=0}E_m(F_t(B_R))= \int_{\partial B_R}\left(v\zeta u_{rr}+u_{rr}u_r\zeta^2-f'(r) u\zeta^2\right) d\sigma.
\end{align}
By \eqref{shougen0} and \eqref{shougen1}, we have $v=-Ru_{rr}(R)\zeta+C$ on $\partial B_R$ , and hence from \eqref{wohewodejiaxiang} and $\eqref{volumepreserving}_1$, we have\begin{align}
\label{wohewodejiaxiang2}
    \frac{d^2}{dt^2}\Big|_{t=0}E_m(F_t(B_R))=& \int_{\partial B_R}\left(-Ru_{rr}^2+u_{rr}u_r-f_r u\right)\zeta^2 d\sigma\nonumber\\
    =&R\int_{\partial B_R}\left( u_{rr}(f+\frac{n}{R}u_r)-\frac{1}{R} f_r u\right)\zeta^2 d\sigma,
\end{align}which is the same as the RHS of \eqref{keyestimate1}. Hence by the exact same argument above, we obtain
\begin{align}
\label{yinqizhuyi}
    \frac{d^2}{dt^2}\Big|_{t=0}E_m(F_t(B_R)) = -R\int_{\partial B_R}\left((f-\frac{n-1}{n}\bar{f}_R)(f-\bar{f}_R)+f_r\bar{f}_R\frac{m}{n^2\omega_nR^{n-1}}\right)\zeta^2d\sigma.
\end{align}
Since $\frac{d^2}{dt^2}\Big|_{t=0}E_m(F_t(B_R))  \ge 0$ and $(f-\frac{n-1}{n}\bar{f}_R)(f-\bar{f}_R)+f_r\bar{f}_R\frac{m}{n^2\omega_nR^{n-1}}$ is a constant on $\partial B_R$, necessarily \begin{align*}
(f-\frac{n-1}{n}\bar{f}_R)(f-\bar{f}_R)+f_r\bar{f}_R\frac{m}{n^2\omega_nR^{n-1}}\le0 \quad \mbox{on $\partial B_R$,}
\end{align*}and the proof is finished.
\end{proof}
\begin{remark}
From \eqref{yinqizhuyi}, it is easy to see that if $f \equiv 1$, then $\frac{d^2}{dt^2}\Big|_{t=0}E_m(F_t(B_R))=0$ if $F_t$ is the flow map generated by constant vector field. This makes sense since $E_m(\Omega)$ is translation invariant.
\end{remark}

Now we are ready to prove Theorem \ref{stabilityintro}.
\begin{proof}[Proof of Theorem \ref{stabilityintro}]
The first and second claim in the corollary follows from the criteria \eqref{sufnecconditionforf}, and the fact that on $\partial B_R$, $f> \bar{f}_R$ when $f$ is nondecreasing and not a constant, and $f \le \bar{f}_R$ when $f$ is nonincreasing. The third claim is also true, because the LHS of \eqref{sufnecconditionforf} is a linear function of $m$ on $\partial B_R$, with negative slope when $f'(R)<0$.
\end{proof}

Also, the following corollary is immediate from \eqref{secondvariationformula} and the proof of Theorem \ref{classification}.
\begin{corollary}
\label{worstcase}
Let $f \ge 0$ be a radial function. Consider
\begin{align*}
    \inf\Big\{\frac{\frac{d^2}{dt^2}\Big|_{t=0}E_m(F_t(B_R))}{\int_{\partial B_R}(\eta\cdot\nu)^2d\sigma}: \mbox{$F_t$ is a smooth volume preserving flow map }\Big\}.
\end{align*}
Then for any $m>0, R>0$, the infimum above is always attained only when $F_t$ is translating flow map with constant velocity. 
\end{corollary}

\section{Shape derivatives of $\lambda_m(\cdot)$ and the precise value of $m_0$ }
Let $\lambda_m(\cdot)$ be defined as in \eqref{lambdam'}. In this section, we will derive the first and second shape derivatives of $\lambda_m(\cdot)$ at the ball shape. Then based on the variation formulas, we will derive Theorem \ref{asympintro}, as a consequences of the translation invariance property of the $\lambda_m$ functional. 

As before, we consider $\Omega_t:=F_t(\Omega)$, where $F_t(x):=F(t,x)$ is the volume-preserving flow map generated by a smooth velocity field $\eta$. Also, note that if  $m>m_0$, then there is a unique minimizer (up to a constant factor) to \eqref{lambdam'} for $\Omega=B_R$, and the minimizer must be radial and positive in $\bar{B}_R$. Let $u(t)(x):=u_{\Omega_t}(x)$ be a function such that $\int_{\Omega_t}u_{\Omega_t}^2\,dx=1$ and the infimum in \eqref{lambdam'} is attained at $u_{\Omega_t}$ when the domain is $\Omega_t$. Standard argument (see for example \cite{DLW}) can show that $u(t)(F_t(\cdot))$ converges to $u_{B_R}(\cdot)$ in $C^0(\bar{B}_R)$ as $t \rightarrow 0$. Hence $u(t)$ is positive on $\partial \Omega_t$ when $|t|$ small.  Given $m>m_0$, since $\lambda_m<\mu_2(B_R)$, we also have that when $|t|$ is small, $\lambda_m(\Omega_t)<\mu_2(\Omega_t)$, where $\mu_2(\Omega_t)$ is the first nonzero eigenvalue of Neumman Laplacian on $\Omega_t$. Hence by looking at the Euler-Lagrange equation (see \cite{BBN}), we conclude that the function $u_{\Omega_t}$ is unique. Hence elliptic regularity theory and implicit function theorem imply that when $m>m_0$, $u(t)(F_t(x))$ is $C^1$ with respect to $t$ for $|t|$ small, and thus $u(t)$ is $C^1$ with respect to $t$ for $|t|$ small, with
\begin{align}
\label{buchong}
    u'(t)(F_t(x))+\nabla u(t) (F_t(x))\cdot \eta(F_t(x))=\frac{d}{dt}\left(u(t)(F_t(x))\right).
\end{align}
We are now ready to derive the first variation of $\lambda_m(\cdot)$ on ball shape.

\begin{proposition}
\label{kun1}
Let $\Omega=B_R$, $\Omega_t=F_t(\Omega)$, where $F_t$ is the volume preserving flow map generated by a smooth vector field $\eta$. Let $u(t)=u_{\Omega_t}$ explained above, and we denote $u_{\Omega}$ by $u$. We also let
\begin{align*}
    \lambda_m(t)=\lambda_m(\Omega_t),
\end{align*}and $\zeta(t)(x)=\eta(x) \cdot \nu(t)$, where $x \in \partial \Omega_t$ and $\nu(t)$ is the unit outer normal on $\partial \Omega_t$. Let $m_0$ be the number where the symmetry breaking of insulating material occurs along $\partial B_R$, then given $m > m_0$, for $|t|$ small, we have
\begin{align}
    \label{eigenfirstvariation}
\lambda_m'(t)=&\int_{\partial \Omega_t} |\nabla u(t)|^2 \zeta(t) dx -2\int_{\partial \Omega_t} |\frac{\partial u(t)}{\partial \nu}|^2 \zeta(t) d\sigma_t\nonumber\\
&+\frac{2}{m}\left(\int_{\partial \Omega_t} u(t)d\sigma_t\right)\int_{\partial \Omega_t} u(t)H(t) \zeta(t) d\sigma_t-\int_{\partial \Omega_t}\lambda_m(t)u^2(t)\zeta(t) d\sigma_t.
\end{align}
\end{proposition}
\begin{proof}
First, let $m>m_0$, and then when $|t|$ is small, $u(t)>0$ on $\partial \Omega_t$. Hence the Euler-Lagrange equation for $u(t)$ is given by
\begin{align}
\label{utequation}
  \begin{cases}
  -\Delta u(t)=\lambda_m(t) u(t) \quad &\mbox{in $\Omega_t$},\\
  \frac{\partial u(t)}{\partial \nu}=-\frac{1}{m}\int_{\partial \Omega_t}u(t) d\sigma_t \quad &\mbox{on $\partial \Omega_t$.}
  \end{cases}
\end{align}
Let $v(t)(x)=u'(t)(x)$. We have
\begin{align*}
    \lambda_m'(t)=&2\int_{\Omega_t}\nabla u(t) \nabla v(t) dx+\int_{\partial \Omega_t} |\nabla u(t)|^2 \zeta(t)d\sigma_t\\
    &+\frac{2}{m}\left(\int_{\partial \Omega_t}ud\sigma_t\right)\int_{\partial \Omega_t}\left(v(t)+\frac{\partial u(t)}{\partial \nu}\zeta(t)+u(t)H\zeta(t)\right)d\sigma_t.
\end{align*}
From \eqref{utequation}, we then have
\begin{align}
\label{yundong0}
    \lambda_m'(t)=&2\lambda_m(t)\int_{\Omega_t}u(t)v(t)dx+\int_{\partial \Omega_t}|\nabla u(t)|^2 \zeta(t) dx\nonumber\\
    &-2\int_{\partial \Omega_t} |\frac{\partial u(t)}{\partial \nu}|^2 \zeta(t) d\sigma_t +\frac{2}{m}\left(\int_{\partial \Omega_t} ud\sigma_t\right)\int_{\partial \Omega_t} u(t)H \zeta(t) d\sigma_t.
\end{align}
Since \begin{align*}
    \int_{\Omega_t} u^2(t)dx\equiv 1,
\end{align*}by taking the derivative we have
\begin{align}
\label{yundong1}
    2\int_{\Omega_t} u(t)v(t)dx+\int_{\partial \Omega_t}u^2(t)\zeta(t) d\sigma_t=0.
\end{align}
Hence combining \eqref{yundong0} and \eqref{yundong1} we have proved \eqref{eigenfirstvariation}.

\end{proof}

\begin{corollary}
\label{stationaryofballeigen}
Let $m_0$ be the number where the symmetry breaking of insulating material occurs along $\partial B_R$. Then for $m > m_0$, $B_R$ is stationary to $\lambda_m(\cdot)$, that is, $\lambda_m'(0)=0$.
\end{corollary}
\begin{proof}
When $m>m_0$, from \eqref{eigenfirstvariation} we know that $\lambda_m'(0)=0$, since $u_{B_R}$ is radial and $F_t$ is volume preserving. Hence $B_R$ is stationary to $\lambda_m(\cdot)$ when $m>m_0$.
\end{proof}

Next, we compute the second shape derivative of the eigenvalue functional on ball shape.
\begin{lemma}
\label{duanqueshengtong}
Let $m> m_0$, then we have
\begin{align}
    \label{eigensecondvariation}
\frac{1}{2}\lambda_m^{''}(0)=\int_{\partial B_R} (u_{rr}v\zeta+u_{rr}u_r\zeta^2)d\sigma+\frac{1}{m}\int_{\partial B_R}u^2 d\sigma\int_{\partial B_R}\left(|\nabla_{\partial B_R}u|^2-\frac{n-1}{R^2}\zeta^2 \right)d\sigma,
\end{align}where $u=u_{B_R}$, $\zeta=\zeta(0)$ and $v=\frac{\partial }{\partial t}\Big|_{t=0} u(t)$.
\end{lemma}
\begin{proof}
We first claim that $v$ satisfies the equation
\begin{align}
\label{eigenvequation}
  \begin{cases}
  -\Delta v=\lambda_m v \quad &\mbox{in $B_R$},\\
  \frac{\partial v}{\partial \nu}=-u_{rr}\zeta-\frac{1}{m}\int_{\partial \Omega} vd\sigma \quad &\mbox{on $\partial B_R$.}
  \end{cases}
\end{align}
Indeed, let $\phi \in C^2(\mathbb{R}^n)$, then 
\begin{align*}
    0=&\frac{d}{dt}\Big|_{t=0}\int_{\Omega_t}(\Delta u(t)+\lambda_m(t) u(t))\phi dx\\
    =& \int_{B_R}(\Delta v+\lambda_m v+\lambda_m'(0)u)\phi dx+\int_{\partial B_R}(\Delta u+\lambda_m u)\phi \zeta dx\\
=& \int_{B_R}(\Delta v+\lambda_m v)\phi dx,\quad \mbox{since Corollary \ref{stationaryofballeigen}.}
\end{align*}This proves $\eqref{eigenvequation}_1$. Similarly,
\begin{align*}
    0=&\frac{d}{dt}\Big|_{t=0}\int_{\partial \Omega_t}\left(\frac{\partial u(t)}{\partial \nu}+\frac{1}{m}\int_{\partial \Omega_t} u(t)d\sigma_t\right)\phi dx\\
    =&\int_{\partial B_R}\left(u_{rr}\zeta+v_r+\frac{1}{m}\int_{\partial B_R}(v+u_r\zeta+uH\zeta)d\sigma\right)\phi d\sigma\\
    =&\int_{\partial B_R}\left(u_{rr}\zeta+v_r+\frac{1}{m}\int_{\partial B_R}vd\sigma\right)\phi d\sigma,
\end{align*}where we have used that $u$ is radial in $B_R$, $\eqref{volumepreserving}_1$ and \eqref{evolutionofarea2}. This proves $\eqref{eigenvequation}_2$.

Now we compute the second derivative of $\lambda_m(\cdot)$ based on \eqref{eigenfirstvariation}.
Taking the derivative of \eqref{eigenfirstvariation} with respect to $t$, by \eqref{volumepreserving}, \eqref{evolutionofarea2}, \eqref{evolutionofnormalspeed} and $\eqref{eigenvequation}_2$ and since $u$ is radial, we have
\begin{align}
\label{diyixiang}
    &\frac{d}{dt}\Big|_{t=0}\int_{\partial \Omega_t}|\nabla u(t)|^2 \zeta(t) d\sigma_t\nonumber\\
    =&\int_{\partial B_R}\left(2( u_rv_r\zeta+u_ru_{rr}\zeta^2)+|\nabla u|^2\zeta( \div \eta-\div_{\partial B_R}\eta)+|\nabla u|^2 \zeta\div_{\partial B_R}\eta\right)d\sigma \nonumber\\
    =&\int_{\partial B_R}\left(u_r\zeta \left(-\frac{1}{m}\int_{\partial B_R}vd\sigma\right)+|\nabla u|^2\zeta \div \eta\right)d\sigma=0,
\end{align}where we have also used that 
\begin{align*}
\nabla u\cdot \nabla v=\partial_{\nu} v \nabla u \cdot \nu=u_rv_r,
\end{align*}
\begin{align*}
    \nabla u  \cdot (\nabla^2 u\,  \eta)=u_r(\nabla^2u:\eta\otimes \nu)=u_ru_{rr}\zeta.
\end{align*}
Similarly as above, we have
\begin{align}
\label{dierxiang}
    &\frac{d}{dt}\Big|_{t=0}\int_{\partial \Omega_t}|\nabla_{\nu} u(t)|^2 \zeta(t) d\sigma_t\\
    =&\int_{\partial B_R} \frac{d}{dt}\Big|_{t=0} \left(\frac{1}{m}\int_{\partial \Omega_t} ud\sigma_t\right)^2\zeta d\sigma\nonumber\\
    &+\int_{\partial B_R}u_r^2\left( \zeta (\div \eta-\div_{\partial B_R}\eta)+\zeta \div_{\partial B_R}\eta)\right)d\sigma=0.
\end{align}
Next, using \eqref{evoonsphere} we have
\begin{align}
\label{disanxiang}
    &\frac{d}{dt}\Big|_{t=0}\int_{\partial \Omega_t}  \left(\frac{2}{m}\int_{\partial \Omega_t} u(t) d\sigma_t \right)u(t)H(t)\zeta(t) d\sigma_t\nonumber\\
=& \int_{\partial B_R}\frac{d}{dt}\Big|_{t=0}  \left(\frac{2}{m}\int_{\partial \Omega_t} u(t) d\sigma_t \right)uH\zeta d\sigma\nonumber\\
&+\left(\frac{2}{m}\int_{\partial B_R} u d\sigma \right)\int_{\partial B_R}\Big(vH\zeta+u_r\zeta H\zeta+u(-\Delta_{\partial B_R} \zeta-\frac{n-1}{R^2}\zeta)\zeta\nonumber\\
&\quad \quad \quad \quad\quad \quad \quad\quad \quad +uH\zeta( \div \eta -\div_{\partial B_R}\eta)+uH\zeta \div_{\partial B_R}\eta\Big)d\sigma\nonumber\\
=& \left(\frac{2}{m}\int_{\partial B_R} u d\sigma \right)\int_{\partial B_R} (vH\zeta+u_rH\zeta^2)d\sigma\nonumber\\
&+ \left(\frac{2}{m}\int_{\partial B_R} u^2 d\sigma \right)\int_{\partial B_R}\left(|\nabla_{\partial B_R} \zeta |^2-\frac{n-1}{R^2}\zeta^2\right)d\sigma.
\end{align}
Last, using that $\lambda_m'(0)=0$ and let $\lambda_m=\lambda_m(0)$, we have
\begin{align}
\label{disixiang}
    &\frac{d}{dt}\Big|_{t=0}\int_{\partial \Omega_t}\lambda_m(t)u^2(t)\zeta(t)d\sigma_t\nonumber\\
    =&\int_{\partial B_R}\left(2\lambda_m uv\zeta+2\lambda_muu_r\zeta^2+\lambda_m u^2 \zeta (\div \eta-\div_{\partial B_R}\eta)+\lambda_m u^2 \zeta \div_{\partial B_R}\zeta\right)d\sigma\nonumber\\
    =&2\lambda_m\int_{\partial B_R}(uv\zeta+uu_r\zeta^2)d\sigma.
\end{align}
Combining \eqref{diyixiang}-\eqref{disixiang} and applying the Euler Lagrange equation of $u$, we have
\begin{align*}
    \frac{1}{2}\lambda_m^{''}(0)=&\int_{\partial B_R}\left((-\lambda_m u-u_r H)v\zeta+(-\lambda_muu_r-u_r^2H)\zeta^2\right)d\sigma\nonumber\\
    &+\frac{1}{m}\int_{\partial B_R}u^2d\sigma\int_{\partial B_R}\left(|\nabla_{\partial B_R} u|^2-\frac{n-1}{R^2}\zeta^2\right)d\sigma \nonumber\\
    =&\int_{\partial B_R}(u_{rr}v\zeta+u_{rr}u_r\zeta^2)d\sigma+\frac{1}{m}\int_{\partial B_R}u^2 d\sigma\int_{\partial B_R}\left(|\nabla_{\partial B_R}u|^2-\frac{n-1}{R^2}\zeta^2 \right)d\sigma.
\end{align*}This proves \eqref{eigensecondvariation}.
\end{proof}
Now we are ready to prove \eqref{buxiedai} in Theorem \ref{asympintro} for any dimension $n \ge 2$.

\begin{proof}[Proof of Theorem \ref{asympintro}]
Recall the definition of $\lambda_m(\Omega)$ defined in \eqref{lambdam'}. If $\Omega=B_R$ and $m>m_0$, then $u_{B_R}$ is unique up to a constant factor. As before, we normalize $u_{B_R}$ such that its total integration over $B_R$ is $1$. In the proof, we write $\lambda_m$ as abbreviation of $\lambda_m(B_R)$. As before, we let $F_t$ be a smooth volume-preserving map, and $u(t):=u_{F_t(B_R)}$ be the function such that $\int_{F_t(B_R)}u(t)^2 dx=1$ and the infimum of \eqref{lambdam'} is achieved at $u(t)$ for $\Omega=F_t(B_R)$. For $|t|$ small, we know that $u(t):=u_{F_t(B_R)}$ is strictly positive on $\partial B_R$ and thus satisfies \eqref{utequation}. Let $u=u(0)$ and $v=u'(0)$, and recall that $v$ satisfies \eqref{eigenvequation}. Since \begin{align*}
    \begin{cases}
-\Delta u=\lambda_m u\quad &\mbox{in $B_R$}\\
\frac{\partial u}{\partial \nu}=-\frac{1}{m}\int_{\partial \Omega}u\quad &\mbox{on $\partial B_R$},
\end{cases}
\end{align*}we know that there exists $w$ satisfies
\begin{align}
\label{shengri}
    \begin{cases}
-\Delta w=\lambda_m w\quad &\mbox{in $B_R$}\\
\frac{\partial w}{\partial \nu}=-u_{rr}\zeta \quad &\mbox{on $\partial B_R$}
\end{cases}
\end{align}Using polar coordinates, it is well known (see for example \cite{nshap}) that the solution $w$ has the form
\begin{align}
    \label{vform}
w(r,\theta) =\sum_{s=0}^{\infty}a_{s,i}r^{1-\frac{n}{2}}J_{s+\frac{n}{2}-1}(\sqrt{\lambda_m}r)Y_{s,i}(\theta),   \quad \theta \in S^{n-1}
\end{align}where $s$ are natural numbers, $i=1,2,\dots, d_s$ for $d_s=(2s+n-2)\frac{(s+n-3)!}{s!(n-2)!}$, $J_s$ are Bessel functions, and $Y_{s,i}$ denotes the $i$-th spherical harmonics of order $s$, that is,
\begin{align*}
    \Delta_{S^{n-1}}Y_{s,i}+s(s+n-2)Y_{s,i}=0\quad \mbox{on $S^{n-1}$.}
\end{align*}
In particular, by choosing $\eta$ to be a constant vector field, then $F_t$ is volume-preserving, and $\zeta$ is a linear combination of $Y_{1,i}=x_i,\, i=1,\dots,n$.  WLOG we let $\zeta=x_1$, then by \eqref{shengri}-\eqref{vform}, we can write $w$ as
\begin{align}
    \label{exprv}
w(r,\theta)=a_1r^{1-\frac{n}{2}} J_{\frac{n}{2}}(\sqrt{\lambda_m}r)x_1,   
\end{align}where $a_1$ is a nonzero constant. By $\eqref{shengri}_2$, we have that
\begin{align}
    \label{lianxi}
a_1\left((1-\frac{n}{2})R^{-\frac{n}{2}}J_{\frac{n}{2}}(\sqrt{\lambda_m}R)+R^{1-\frac{n}{2}}\sqrt{\lambda_m}J_{\frac{n}{2}}'(\sqrt{\lambda_m}R)\right) x_1=-u_{rr}(R)x_1.
\end{align}
Since 
\begin{align}
\label{ur}
    u_r(R)=-\frac{1}{m}\int_{\partial B_R}ud\sigma=-\frac{P(B_R)u(R)}{m},
\end{align}from the equation of $u$ we have that
\begin{align}
    \label{urr}
u_{rr}(R)=-\lambda_m u(R)-\frac{n-1}{R}u_r(R)=\left(\frac{n-1}{R}\frac{P(B_R)}{m}-\lambda_m\right)u(R).
\end{align}Hence \eqref{lianxi} and \eqref{urr} lead to
\begin{align}
\label{shenqi}
    a_1\left((1-\frac{n}{2})R^{-\frac{n}{2}}J_{\frac{n}{2}}(\sqrt{\lambda_m}R)+R^{1-\frac{n}{2}}\sqrt{\lambda_m}J_{\frac{n}{2}}'(\sqrt{\lambda_m}R)\right)=\left(\lambda_m-\frac{n-1}{R}\frac{P(B_R)}{m}\right)u(R).
\end{align}
For $\zeta=x_1$, since $w$ is actually of the form $v+Cu$ for some constant $C$, and since $u$ is radial, by $\eqref{volumepreserving}_1$ and \eqref{eigensecondvariation}, we have
\begin{align*}
    \frac{1}{2}\lambda_m^{''}(0)=\int_{\partial B_R}(u_{rr}w\zeta+u_{rr}u_r\zeta^2)d\sigma.
\end{align*}
Since $\lambda_m(\cdot)$ does not depend on translation of the domain, we have $\lambda_m^{''}(0)=0$ for our choice of $F_t$ and $\zeta$. Hence 
\begin{align}
\label{xunsi}
    \int_{\partial B_R}(u_{rr}w\zeta+u_{rr}u_r\zeta^2)d\sigma=0.
\end{align}
By \eqref{exprv}, and \eqref{ur}-\eqref{xunsi}, we have
\begin{align}
\label{yibangongshi}
   0=&\int_{\partial B_R}\left(\frac{n-1}{R}\frac{P(B_R)}{m} -\lambda_m\right)\cdot\nonumber\\
   &\left(\frac{\lambda_m-\frac{n-1}{R}\frac{P(B_R)}{m}}{(1-\frac{n}{2})R^{-\frac{n}{2}}J_{\frac{n}{2}}(\sqrt{\lambda_m}R)+R^{1-\frac{n}{2}}\sqrt{\lambda_m}J_{\frac{n}{2}}'(\sqrt{\lambda_m}R)}R^{1-\frac{n}{2}}J_{\frac{n}{2}}(\sqrt{\lambda_m}R)-\frac{P(B_R)}{m}\right)u^2x_1^2d\sigma.
\end{align}
Since $m\lambda_m$ is strictly increasing, for all $m>m_0$ except at most one point, we have
\begin{align}
\label{shounuo}
    \frac{m\lambda_m-\frac{n-1}{R}P(B_R)}{(1-\frac{n}{2})R^{-\frac{n}{2}}J_{\frac{n}{2}}(\sqrt{\lambda_m}R)+R^{1-\frac{n}{2}}\sqrt{\lambda_m}J_{\frac{n}{2}}'(\sqrt{\lambda_m}R)}R^{1-\frac{n}{2}}J_{\frac{n}{2}}(\sqrt{\lambda_m}R)=P(B_R).
\end{align}
By continuity, \eqref{shounuo} holds for every $m>m_0$. Now we let $m \rightarrow m_0^+$, then $R\sqrt{\lambda_m}$ converges to $R\sqrt{\mu_2(B_R)}=\mu_2(B_1)$, which is the first positive root of the following equation
\begin{align*}
    zJ_{\frac{n}{2}}'(z)-\frac{n-2}{2}J_{\frac{n}{2}}(z)=0.
\end{align*}
Then as $m\rightarrow m_0^+$, the denominator in \eqref{shounuo} goes to $0$, which forces that 
\begin{align*}
    m_0\lambda_{m_0}-\frac{n-1}{R}P(B_R)=0
\end{align*}
Hence $$m_0\lambda_{m_0}=(n-1)n\omega_nR^{n-2}=\frac{n-1}{n}\frac{P^2(B_R)}{|B_R|}.$$
This is exactly \eqref{buxiedai} since $\lambda_{m_0}=\mu_2$.
\end{proof}

\begin{remark}
For $n=2$, \eqref{buxiedai} can also be obtained by letting $m \rightarrow m_0$ in \eqref{shenqi}, without referring to the second shape derivative of $\lambda_m$ at ball shape. 
\end{remark}

The following corollary is immediate from \eqref{buxiedai}.
\begin{corollary}
\label{exactformula}
Let $B_R$ be the disk of radius $R$ in $\mathbb{R}^2$ and $m_0(R)$ be the number where the symmetry breaking of insulating material around $\partial B_R$. Then $m_0(R)$ has the exact formula as
\begin{align}
\label{moR}
    m_0(R)=\frac{n^2}{2}\frac{|B_R|}{\mu_2(B_1)}\approx \frac{2\pi}{3.39}R^2.
\end{align}
\end{corollary}
\begin{remark}
\eqref{moR} is interesting because it says that the symmetry breaking point at which the symmetry of insulating material around $\partial B_R$ breaks, is in fact proportional to the volume of $B_R$, instead of the perimeter of $B_R$.
\end{remark}

The following corollary of Theorem \ref{asympintro} gives another way of understanding the limit of $m\lambda_m$ as $m \rightarrow \infty$.
\begin{corollary}
\begin{align}
\label{4pi}
    \lim_{m\rightarrow \infty}m\lambda_m(B_R)=\frac{P^2(B_R)}{|B_R|}.
\end{align}
\end{corollary}
\begin{proof}
As $m\rightarrow \infty$, $\sqrt{\lambda_m}R \rightarrow 0$. Then from the fact that
\begin{align*}
    \lim_{t \rightarrow 0}\frac{tJ_s'(t)}{J_s(t)}=s,
\end{align*}\eqref{shounuo} implies 
\begin{align*}
    \frac{\lim_{m \rightarrow \infty}m\lambda_m(B_R)-\frac{n-1}{R}P(B_R)}{(1-\frac{n}{2})R^{-\frac{n}{2}}+R^{-\frac{n}{2}}\frac{n}{2}}R^{1-\frac{n}{2}}=P(B_R)
\end{align*}
This implies \eqref{4pi}.
\end{proof}

\section{stability of $B_R$ for $m>m_0$ in the eigenvalue problem}
In this section, we will prove that when $n=2$, $B_R$ is a stable solution to $\lambda_m(\cdot)$ for any $m>m_0$, where $m_0=\frac{2\pi R^2}{\mu_2(B_1)}\approx \frac{2\pi}{3.39}R^2$ is the symmetry breaking number for $B_R$, due to Corollary \ref{exactformula}. 

Before proving this, we need the following lemma on Bessel functions.
\begin{lemma}
\label{dandiaoxing}
For a fixed $t\in (0,j_1')$, where $j_1'$ is the first zero of $J_1'$, we have
\begin{align*}
    \frac{d}{ds}\left(\frac{J_s(t)}{tJ_s'(t)}\right)<0.
\end{align*}
\end{lemma}
\begin{proof}
Using properties of Bessel functions, we have
\begin{align*}
    \frac{J_s(t)}{tJ_s'(t)}=\frac{J_s(t)}{t\left(-J_{s+1}+\frac{sJ_s(t)}{t}\right)}=\frac{1}{-t\frac{J_{s+1}(t)}{J_s(t)}+s}.
\end{align*}
By \cite[Theorem 2]{Landau}, 
\begin{align*}
    t\frac{d}{ds}\left(\frac{J_s(t)}{J_{s+1}(t)}\right) \ge 2.
\end{align*}
Hence fixing $t \in (0,j_1')$, $\frac{J_s(t)}{J_{s+1}(t)}$ is a strictly increasing function with respect to $s$, and hence $ \frac{J_s(t)}{tJ_s'(t)}$ is a strictly decreasing function with respect to $s$.
\end{proof}

Now we prove can finish the proof of Theorem \ref{asympintro}.
\begin{proposition}
\label{stableformlarge}
Let $n=2$, then $B_R$ is stable to $\lambda_m(\cdot)$ for any $m>m_0$.
\end{proposition}
\begin{proof}
Let $F_t$ be the volume-preserving map generated by a smooth vector field $\eta$, and let $\zeta=\eta \cdot \nu$ on $\partial B_R$. Using previous notations, recall that we have proved in Lemma \ref{duanqueshengtong} the following second variation formula \begin{align}
\label{shijianguanli}
\frac{1}{2}\lambda_m^{''}(0)=\int_{\partial B_R} (u_{rr}v\zeta+u_{rr}u_r\zeta^2)d\sigma+\frac{1}{m}\int_{\partial B_R}u^2 d\sigma\int_{\partial B_R}\left(|\nabla_{\partial B_R}u|^2-\frac{n-1}{R^2}\zeta^2 \right)d\sigma.
\end{align}
Given such $\zeta$, by the proof of Theorem \ref{asympintro}, there is a solution $w$ to \eqref{shengri}. As in the proof, $w$ is actually obtained as $v+Cu$ for some constant $C$. Also, since $\lambda_m(B_R)<\mu_2(B_R)$ as $m$ large, such solution $w$ is unique. From \eqref{shijianguanli}, and since $\eqref{volumepreserving}_1$ and $u$ is radial, we have
\begin{align}
    \label{eigensecondvariationw}
\frac{1}{2}\lambda_m^{''}(0)=\int_{\partial B_R} (u_{rr}w\zeta+u_{rr}u_r\zeta^2)d\sigma+\frac{1}{m}\int_{\partial B_R}u^2 d\sigma\int_{\partial B_R}\left(|\nabla_{\partial B_R}u|^2-\frac{n-1}{R^2}\zeta^2 \right)d\sigma.
\end{align}
Again by $\eqref{volumepreserving}_1$, we may write the Fourier series of $\zeta$ on $\partial B_R$ as
\begin{align}
\label{fourierseriesforzeta}
    \zeta=\sum_{s\ge 1}(c_s\cos s\theta+d_s\sin s\theta).
\end{align}
Hence from the equation of $w$, and by writing $w$ in terms of polar form, similarly as before we can obtain
\begin{align*}
    w(r,\theta)=&-u_{rr}(R)\sum_{s \ge 1}\frac{J_s(\sqrt{\lambda_m}r)}{\sqrt{\lambda_m}J_s'(\sqrt{\lambda_m}r)}(c_s\cos s\theta+d_s\sin s\theta)\\
    =&u(R)\sum_{s\ge1}\frac{m\lambda_m-2\pi}{m\sqrt{\lambda_m}J_s'(\sqrt{\lambda_m}R)}J_s(\sqrt{\lambda_m}r)(c_s\cos s\theta+d_s\sin s\theta),
\end{align*}where the last equality is from \eqref{urr}.
Substituting $w$ given by the above, $u_r(R)=-2\pi Ru(R)/m$, $u_{rr}(R)$ given by \eqref{urr} and $\zeta$ given by \eqref{fourierseriesforzeta} into the second variation formula \eqref{eigensecondvariationw}, direct computation yields
\begin{align*}
    \frac{1}{2}\lambda_m^{''}(0)=&\frac{R}{m}(\frac{2\pi}{m}-\lambda_m)\pi u^2(R)\sum_{s\ge1}f_s(c_s^2+d_s^2)\nonumber\\
    &+\frac{1}{m}\int_{\partial B_R}u^2 d\sigma\int_{\partial B_R}\left(|\nabla_{\partial B_R}u|^2-\frac{n-1}{R^2}\zeta^2 \right)d\sigma,
\end{align*}where
\begin{align}
    \label{fs}
f_s=\frac{m\lambda_m-2\pi}{(\sqrt{\lambda_m}R)J_s'(\sqrt{\lambda_m}R)}J_s(\sqrt{\lambda_m}R)-2\pi.
\end{align}
Note that $f_s=0$ for $s=1$, which is exactly \eqref{shounuo} for $n=2$. Hence from here and by Lemma \ref{dandiaoxing}, we have that when $m>m_0$ and $s \ge 2$, $f_s<f_1=0$.  Since $\frac{2\pi}{m}-\lambda_m<0$ when $m>m_0$, due to \eqref{buxiedai} in two dimensions and the fact that $m\lambda_m$ is strictly increasing, we therefore have that $\lambda^{''}(0) \ge 0$. This really says that $B_R$ is stable when $m>m_0$.
\end{proof}

\end{document}